\documentclass[10pt,a4paper]{article}
\usepackage{fullpage}
\usepackage{theorem,graphicx,amssymb,amsmath,url}

\usepackage{theorem,graphicx,amssymb,amsmath,mathrsfs,url}

\theorembodyfont{\slshape}

\newtheorem{thm}{Theorem}
\newtheorem{lem}[thm]{Lemma}
\newtheorem{cor}[thm]{Corollary}
\newtheorem{prop}[thm]{Proposition}
\newtheorem{defn}[thm]{Definition}
\newtheorem{rmk}[thm]{Remark}

\def\QED{\ensuremath{{\square}}}
\def\markatright#1{\leavevmode\unskip\nobreak\quad\hspace*{\fill}{#1}}
\newenvironment{proof}
{\begin{trivlist}\item[\hskip\labelsep{\bf Proof.}]}
	{\markatright{\QED}\end{trivlist}}

\usepackage[mathlines]{lineno}
\usepackage{xcolor}
\usepackage[inline]{enumitem}
\usepackage{bbm}
\usepackage{hyperref} 

\DeclareMathOperator{\Ima}{Im}

\graphicspath{{./figures/}}

\begin{document}
\title{A note on friezes of type $\Lambda_4$ and $\Lambda_6$}
\author{
    Lukas Andritsch\thanks{Mathematics and Scientific Computing, University of Graz, Graz, Austria, \newline  {\tt lukas.andritsch@uni-graz.at} }
}

\title{A note on friezes of type $\Lambda_p$}

	\author{
		Lukas Andritsch\thanks{Mathematics and Scientific Computing, University of Graz, Graz, Austria, \newline  {\tt lukas.andritsch@uni-graz.at} }}

\maketitle


\begin{abstract}
	A frieze is an array of numbers obeying the unimodular rule. Coxeter showed that a frieze with integer entries corresponds to a triangulation. Recently, Holm and J{\o}rgenson introduced friezes of type $\Lambda_p$ which correspond to $p$-angulations of a polygon. In this paper we explore the connection between these two types of friezes; in particular we show that the friezes of type $\Lambda_p$ for $p=4$ and $p=6$ contain integral friezes within them. We also consider the relationships with Farey graphs.
\end{abstract}

	

\tableofcontents

\section{Introduction}
\nocite{bpt_2016} \nocite{c_2015} \nocite{CoCo_1973a} \nocite{CoCo_1973b} \nocite{Co_1971} \nocite{dp_1993} \nocite{hj_2017} \nocite{m_2015} \nocite{m-got_2015} \nocite{ps_2000} \nocite{sw_2016}
Friezes (or frieze patterns) were defined and studied by Coxeter in \cite{Co_1971}. These are arrays of the following form:

\begin{defn}\label{no:frieze}
	A frieze $F$ of width $n$ is an array of shifted infinite rows of (positive) real numbers, satisfying the unimodular rule: for every diamond of the form 
	$\begin{smallmatrix} & b \\ a & & d \\ & c \end{smallmatrix}$, we have $ad-bc=1$. 
	Furthermore, it starts and ends with rows of 0's and 1's and has $n$ non-trivial rows between them. 
\end{defn}

See Figure~\ref{fig:cc_frieze} for an example of a frieze of width 7. 
We subscript the entries so that the $r$-th row has entries $m_{i,i+r}$ for $i\in \mathbb{Z}$ and $-2\leq r \leq n+1$. The trivial rows of $0$'s at top and bottom have entries $m_{i,i-2}$ and $m_{i,i+n+1}$ and the trivial rows of 1's have entries $m_{i,i-1}$ and $m_{i,i+n}$. Thus for all $i\in \mathbb{Z}$ and $-1\leq r\leq n$, the unimodular rule is $m_{i,i+r}m_{i+1,i+r+1}-m_{i+1,i+r}m_{i,i+r+1}=1.$
\begin{rmk}
	When we look at an entry $m_{i,j}$ of a frieze $(m_{i,j})_{i,j\in \mathbb{Z}}$ of width $n$, we tacitly require $i\in \mathbb{Z}$ and $-2\leq j-i \leq n+1$.  
\end{rmk}

Due to~\cite{Co_1971}, rows of friezes of width $n$ are periodic with period dividing $n+3$. The first non-trivial row $(m_{i,i})_{i\in \mathbbm{Z}}$ is called the \emph{quiddity row} and we call any $(n+3)$-tuple of successive entries of the quiddity row a {\em quiddity sequence}. All entries of a given frieze are determined by any of the frieze's quiddity sequences.

\begin{figure}[htp!]
\begin{center}	\resizebox{0.8\linewidth}{!}{
		\begin{tabular}{c c c c c c c c c c c c c c c c c c c c c}
			$0$ & & $0$&  & $0$  & & $0$ & & $0$ & & $0$ & & $0$ & & $0$ & & $0$ & & $0$ & &  \\
			&$1$ & &$1$  &  &$1$ & &$1$ & &$1$ & &$1$ &  &$1$ & &$1$ & &$1$ & &$1$ & \\	
			$1$ & & $4$&  & $1$  & & $2$ & & $3$ & & $4$ & & $1$ & & $2$ & & $2$ & & $4$ & &  \\
			&$3$ & &$3$  &  &$1$ & &$5$ & &$11$ & &$3$ &  &$1$ & &$3$ & &$7$ & &$3$ & \\		
			$8$ & & $2$ & &  $2$&  & $2$  & & $18$ & & $8$ & & $2$ & & $1$ & & $10$ & & $5$ & &    \\ 
			&$5$&  &$1$ & &$3$  &  &$7$ & &$13$ & &$5$ & &$1$ &  &$3$ & &$7$ & &$13$ &  \\		
			$8$ & & $2$ & &  $1$&  & $10$  & & $5$ & & $8$ & & $2$ & & $2$ & & $2$ & & $18$ & &    \\
			&$3$& & $1$ & &$3$ & &$7$  &  &$3$ & &$3$ & &$3$ & &$1$ &  &$5$ & &$11$  & \\		
			$4$ & & $1$ & &  $2$&  & $2$  & & $4$ & & $1$ & & $4$ & & $1$ & & $2$ & & $3$ & &    \\
			&$1$ & &$1$  &  &$1$ & &$1$ & &$1$ & &$1$ &  &$1$ & &$1$ & &$1$ & &$1$ & \\	
			$0$ & & $0$&  & $0$  & & $0$ & & $0$ & & $0$ & & $0$ & & $0$ & & $0$ & & $0$ & & \\ 	
		\end{tabular} }  
		\caption{A frieze of width $7$ with quiddity sequence $(1,4,1,2,3,4,1,2,2,4)$.}\label{fig:cc_frieze}
	\end{center}
	\end{figure}

	We will use the following dependencies of the entries of a frieze and its quiddity row in further sections:
	\begin{align}\label{eq:relation_1}
	m_{i,i+r}=m_{i,i}m_{i+1,i+r}-m_{i+2,i+r},
	\end{align}
	and
	\begin{align}\label{eq:relation_2}
	m_{i,i+r}=m_{0,i-2}m_{1,i+r}-m_{0,i+r}m_{1,i-2}
	\end{align}
	stated in \cite{bpt_2016} and \cite{m-got_2015} respectively.
	
	Conway and Coxeter provided a bijection between triangulations of (regular and convex) polygons with $n+3$ vertices and integral friezes of width $n$ in \cite{CoCo_1973a} and \cite{CoCo_1973b}. These friezes are therefore counted by the Catalan numbers and we call them \textit{Conway--Coxeter friezes}, see Figure~\ref{fig:cc_frieze} for an example. 
	
	Friezes have been generalised in several ways, in particular due to connections with the theory of cluster algebras, see \cite{m_2015} for a survey. Holm and J{\o}rgensen recently generalised in \cite[Definition 0.2.]{hj_2017} the relation between friezes and polygonal dissections by studying certain friezes over $\mathcal{O}_K$, the ring of algebraic integers in the field $K=\mathbbm{Q}(\lambda_{p_1},\ldots, \lambda_{p_s})$, for $s$ and $\lambda_{p_i}$, $1\leq i \leq s$ as cited below.
	
	\begin{defn}
		Let $p \geq 3$ be an integer. A frieze is of type $\Lambda_p$ if the quiddity row consists of (necessarily) positive integral multiples of
		\begin{align*}
		\lambda_p=2 \cos\left(\frac{\pi}{p} \right).
		\end{align*}
	\end{defn}
	
	\begin{defn}
		A $p$-angulation $D$ of a (convex, regular) polygon $P$ is a set of non-intersecting diagonals that partition $P$ into a finite collection of $p$-gons.  	
	\end{defn}	
	
	\begin{rmk}\label{rmk:vertices_of_p-angulation} 
		There exist $p$-angulations of a polygon $P$ if and only if the number of vertices of $P$ is $(p\!-\!2)s\!+2$ for some $s\geq 1$. The number of $p$-angulations of $P$ is counted by the Fuss--Catalan number $\tfrac{1}{s}{(p-2)\cdot s+s \choose s-1}$, going back to Fuss--Euler (cf.~\cite{ps_2000}). 
	\end{rmk}
	One of the main results of \cite{hj_2017} is the following theorem:
	
	\begin{thm}
		There is a bijection between $p$-angulations of an $(n\!+\!3)$-gon and friezes of type $\Lambda_p$ and width $n$.
	\end{thm}
	
	A $p$-angulation of $P$, which splits $P$ into $p$-gons $P_1$, $\ldots$, $P_s$ is mapped to a frieze $F$ constructed as follows. To each vertex $\alpha$ of $P$, we associate
	\begin{align*}
	q_{\alpha} := | \{P_j \mid P_j \text{ is incident with } \alpha \}|.
	\end{align*}
	The quiddity sequence of a frieze of a $p$-angulation of a $(p\!-\!2)s\!+2$-gon is
	\begin{align}\label{eq:quiddity_p-angulation}
	(\lambda_p q_0,\lambda_p q_1,\ldots, \lambda_p q_{(p-2)s+1}).
	\end{align}
	Note that $q_{\alpha} = \deg(\alpha)-1$, where $\deg(\alpha)$ is the number of edges that are incident to the vertex $\alpha$.
	
	Figure~\ref{fig:quadrangulation} shows a $4$-angulation with diagonals $\{1,4\}$, $\{4,9\}$ and $\{5,8\}$. Therefore $q_0=q_2=q_3=q_6=q_7=1$, $q_1=q_5=q_8=q_9=2$ and $q_4=3$. Figure~\ref{fig:lambda_4_frieze} shows the associated frieze of type $\Lambda_4$.
	\begin{figure}[!htb]
		\begin{center}
			\includegraphics[scale=0.35, page=1]{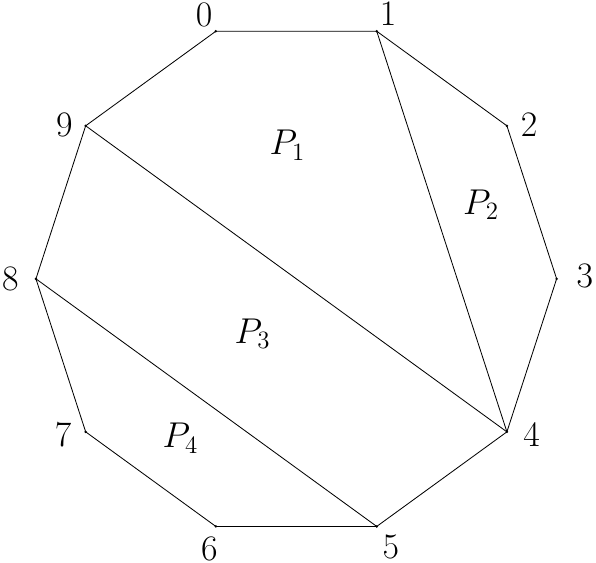}
			\caption{A $4$-angulation of a $10$-gon.}\label{fig:quadrangulation}  
		\end{center}
	\end{figure} 
	\begin{figure}
		\resizebox{\linewidth}{!}{\begin{tabular}{c c c c c c c c c c c c c c c c c c c c c}
				$0$ & & $0$&  & $0$  & & $0$ & & $0$ & & $0$ & & $0$ & & $0$ & & $0$ & & $0$ & &  \\
				&$1$ & &$1$  &  &$1$ & &$1$ & &$1$ & &$1$ &  &$1$ & &$1$ & &$1$ & &$1$ & \\	
				$\sqrt{2}$ & & $2\sqrt{2}$&  & $\sqrt{2}$  & & $\sqrt{2}$ & & $3\sqrt{2}$ & & $2\sqrt{2}$ & & $\sqrt{2}$ & & $\sqrt{2}$ & & $2\sqrt{2}$ & & $2\sqrt{2}$ & &  \\
				&$3$ & &$3$  &  &$1$ & &$5$ & &$11$ & &$3$ &  &$1$ & &$3$ & &$7$ & &$3$ & \\	
				$4\sqrt{2}$ & & $2\sqrt{2}$ & &  $\sqrt{2}$&  & $2\sqrt{2}$  & & $9\sqrt{2}$ & & $8\sqrt{2}$ & & $\sqrt{2}$ & & $\sqrt{2}$ & & $5\sqrt{2}$ & & $5\sqrt{2}$ & &    \\ 
				&$5$&  &$1$ & &$3$  &  &$7$ & &$13$ & &$5$ & &$1$ &  &$3$ & &$7$ & &$13$ &  \\			
				$8\sqrt{2}$ & & $\sqrt{2}$ & &  $\sqrt{2}$&  & $5\sqrt{2}$  & & $5\sqrt{2}$ & & $4\sqrt{2}$ & & $2\sqrt{2}$ & & $\sqrt{2}$ & & $2\sqrt{2}$ & & $9\sqrt{2}$ & &    \\
				&$3$& & $1$ & &$3$ & &$7$  &  &$3$ & &$3$ & &$3$ & &$1$ &  &$5$ & &$11$  & \\	
				$2\sqrt{2}$ & & $\sqrt{2}$ & &  $\sqrt{2}$&  & $2\sqrt{2}$  & & $2\sqrt{2}$ & & $\sqrt{2}$ & & $2\sqrt{2}$ & & $\sqrt{2}$ & & $\sqrt{2}$ & & $3\sqrt{2}$ & &    \\
				&$1$ & &$1$  &  &$1$ & &$1$ & &$1$ & &$1$ &  &$1$ & &$1$ & &$1$ & &$1$ & \\	
				$0$ & & $0$&  & $0$  & & $0$ & & $0$ & & $0$ & & $0$ & & $0$ & & $0$ & & $0$ & & \\ 	
			\end{tabular}}  
			\caption{The frieze of type $\Lambda_4$ corresponding to the $4$-angulation of Figure~\ref{fig:quadrangulation} has width $7$.}\label{fig:lambda_4_frieze}
		\end{figure}
		
		In the following two sections, we consider $p$-angulations for $p=4$ and $p=6$ and prove the main theorem (Proposition~\ref{prop:friezes}, Proposition~\ref{prop:uniqueness} and discussion in Section~\ref{sec:friezes_type_6}):
		\begin{thm}\label{thm:main}
			Let $p=4$ or $p=6$. For every $p$-angulation $D$ of a polygon $P$, there exist exactly two triangulations $T_D^{\circ}$ and $T_D^{\bullet}$ such that the frieze of type $\Lambda_p$ associated to $D$ and the Conway--Coxeter friezes of $T_D^{\circ}$ and $T_D^{\bullet}$ coincide in every second row.
		\end{thm}
		
		Hence we give a partial answer to question (i) of \cite{hj_2017}, which asks whether there is a useful characterisation of friezes of type $\Lambda_p$ for $p\geq 4$: We show that there exist triangulations, whose associated Conway--Coxeter friezes recover the friezes of type $\Lambda_p$ of a given $p$-angulation for $p=4$ or $p=6$. Lemma~\ref{lem:power_of_lambda_p} at the end of Section~\ref{sec:friezes_type_6} shows that Theorem~\ref{thm:main} cannot hold for any other value of $p$. 
		
		Finally, the last section deals with associated Farey graphs to friezes of type $\Lambda_p$ for general $p$, which is a generalisation of results of \cite{m-got_2015} for $p=3$. We give a formula for the entries of the frieze of type $\Lambda_p$ in terms of values of the vertices of the corresponding $p$-angulated path in the Farey graph and vice versa. Combining these formulas with the results of previous sections gives direct links between Farey graphs of Conway--Coxeter friezes and Farey graphs of friezes of type $\Lambda_4$ and $\Lambda_6$.

		\section{Friezes of type $\Lambda_4$ and associated Conway--Coxeter friezes.}\label{sec:friezes_type_4}
		There is a bijection between quadrangulations on the $2s+2$-gon  and noncrossing trees on $s+1$ vertices as defined in \cite{dp_1993}.

		We follow \cite{c_2015} for describing this bijection. Assume that the vertices of the $2s+2$-gon have been coloured black and white alternating, with odd vertices coloured black and even vertices coloured white. 
		Then every edge of the quadrangulation connects one black and one white vertex. Every quadrangle contains a unique black--black diagonal. The collection of these diagonals form a noncrossing tree. Conversely, by drawing a noncrossing tree using the black vertices of the $2s+2$-gon, one can consider all black--white edges that do not cross the edges of the noncrossing tree. This gives back the quadrangulation. We can define another bijection between trees and quadrangulations in the same way by considering the white vertices (and white--white diagonals) of the quadrangulation. These constructions are shown in Figure~\ref{fig:bij_trees_quadrangulations}.
		
		\begin{figure}[!htb]
			\begin{center}
				\includegraphics[scale=0.33, page=2]{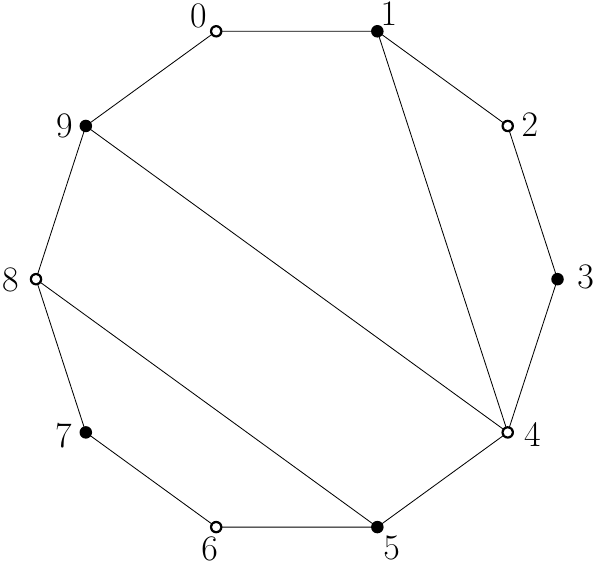}
				\hspace{5 ex}
				\includegraphics[scale=0.33, page=3]{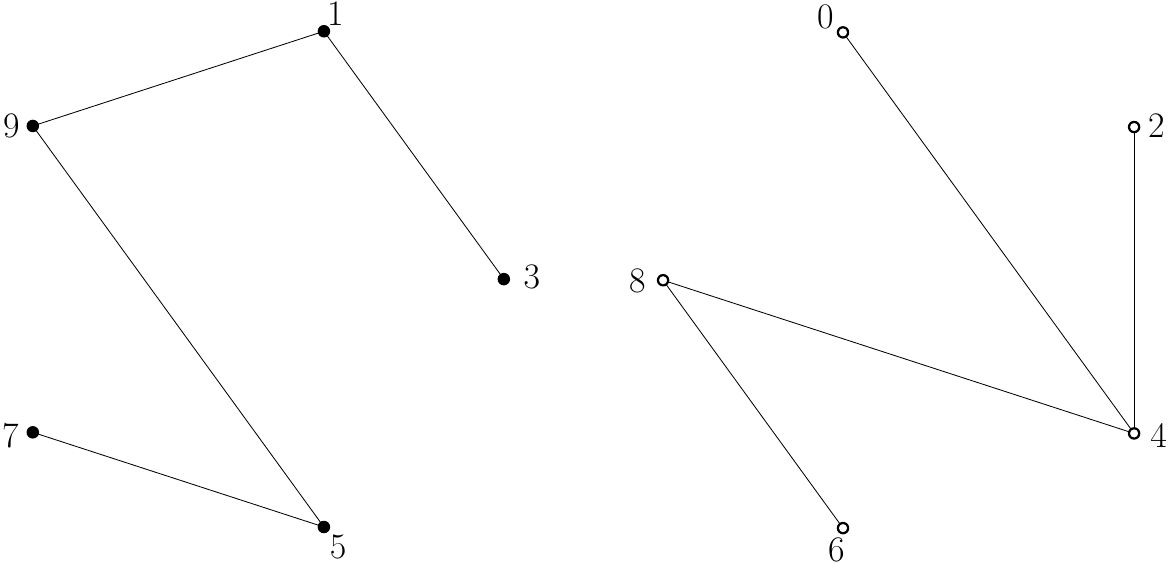}
				\caption{Two bijections between quadrangulations and noncrossing trees.}\label{fig:bij_trees_quadrangulations} 
			\end{center}
		\end{figure}
		
		\begin{defn}\label{def:associated_triangulation}	
			To every quadrangulation $D$ of a polygon $P$ we associate the triangulation $T_D^{\bullet}$ consisting of diagonals defined by the quadrangulation $D$ of $P$ together with the edges of the corresponding noncrossing tree on the black vertices (and similarly $T_D^{\circ}$ with the edges on the white vertices instead).
		\end{defn}
		
		Figure~\ref{fig:triangulation_quadrangulation} shows the associated triangulation of quadrangulation of Figure~\ref{fig:quadrangulation}.
		\begin{figure}
			\begin{center}
				\includegraphics[scale=0.35, page=5]{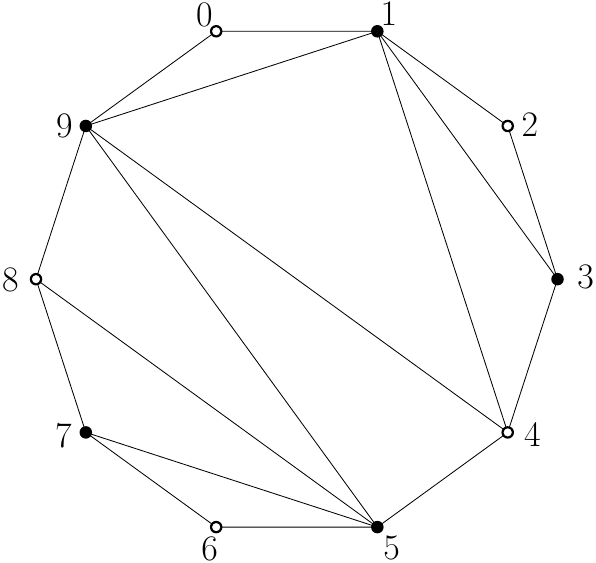}
				\hspace{5 ex}
				\includegraphics[scale=0.35, page=6]{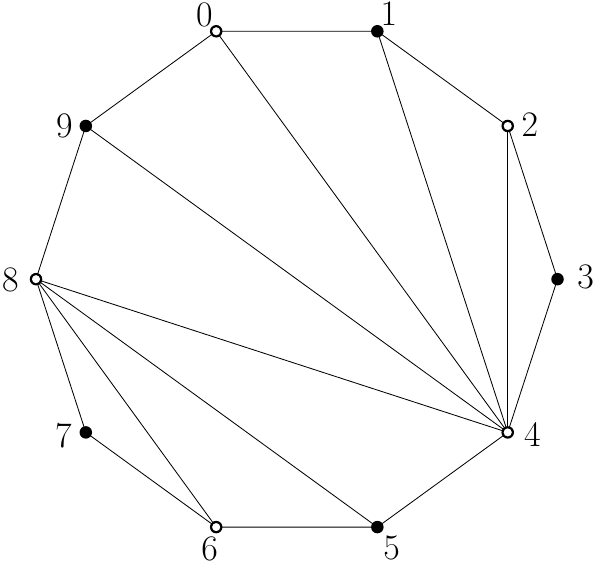}
				\caption{The triangulation $T_{D}^{\bullet}$ (left) and $T_D^{\circ}$ (right) of the quadrangulation $D$ in Figure~\ref{fig:quadrangulation}.}\label{fig:triangulation_quadrangulation} 
			\end{center}
		\end{figure} 	
		For the triangulation $T_{D}^{\bullet}$ associated to a quadrangulation $D$, define for every vertex  $0 \leq \alpha \leq 2s+1$,
		\begin{align*}
		t_{\alpha}^{\bullet} := | \{P_j \mid P_j \text{ is incident with } \alpha \}|.
		\end{align*}
		The quiddity sequence of the Conway--Coxeter frieze is then by definition \newline $(t_0^{\bullet},t_1^{\bullet}, \ldots,t_{2s+1}^{\bullet})$. Figure~\ref{fig:cc_frieze} shows the Conway--Coxeter frieze of the triangulation $T_D^{\bullet}$ of Figure~\ref{fig:triangulation_quadrangulation}. Note that every second row coincides with the corresponding row of the frieze of type $\Lambda_4$ in Figure~\ref{fig:lambda_4_frieze}.
		
		In order to prove the following lemma, we need an equivalent combinatorial object to a $p$-angulation and its properties.
		
		The {\em dual graph} of a $p$-angulation $D$ of a polygon $P$ has a vertex for every subpolygon $P_j$ and two vertices are connected by an edge if the corresponding $p$-gons share a common edge. An {\em ear} is a $p$-gon $P_j$ which consists of $p\!-\!1$ boundary edges of $P$ and exactly one diagonal of $D$.
		
		The dual graph of any $p$-angulation $D$ of a polygon $P$ is a tree. Every tree has at least two leaves and every leaf of the dual graph of $D$ corresponds to an ear of $D$.
		
		Recall that for a quadrangulation $D$ of a polygon into $s$ quadrilaterals,  $q_{\alpha}$ denotes the number of quadrilaterals incident with vertex $\alpha$ for $0\leq \alpha \leq 2s+1$. We can now state the following lemma:
		\begin{lem}\label{lem:number_triangles}
			For every $4$-angulation $D$ and its associated triangulation $T_D^{\bullet}$, the following equalities hold:
			\begin{align}\label{eq:faces_vertices}
			t_{\alpha}^{\bullet} = \left\{ \begin{matrix}
			q_{\alpha} & \text{ if } \alpha \text{ is even} \\
			2q_{\alpha} & \text{ otherwise.}
			\end{matrix} \right.
			\end{align} 	
		\end{lem}
		\begin{proof}
			Let the number of vertices of the polygon be $2s+2$. 
			
			The proof is done by induction on $s$.
			For $s=1$, the statement is true. 
			
			Let $s>1$ and let $D$ be a quadrangulation of the polygon. Let $\{\alpha,\alpha+3\}$ be an ear of $D$. Cutting off this ear from $D$, as shown in Figure~\ref{fig:cutting_quadrilateral}, gives a quadrangulation $D'$ with $2s$ vertices and hence equation \eqref{eq:faces_vertices} holds for all vertices $\beta \in [0,2s+1]\setminus\{\alpha,\alpha+1,\alpha+2,\alpha+3\}$ of $T_D^{\bullet}$ by the induction hypothesis.
			
			Denote the number of triangles and quadrilaterals incident with vertices $\alpha$ and $\alpha+3$ in $T_{D'}^{\bullet}$ and $D'$ by $t_{\alpha}'$, $q_{\alpha}'$ and $t_{\alpha+3}'$, $q_{\alpha+3}'$   respectively.  There are 2 cases, as vertex $\alpha$ is either coloured black or white. If $\alpha$ is odd and hence coloured black, then $t_{\alpha}'= 2q_{\alpha}'$ and $t_{\alpha+3}'=q_{\alpha+3}'$. As the triangulation $T_D^{\bullet}$ includes the diagonal $\{\alpha,\alpha+2\}$, we obtain the equations
			\begin{align*}
			t_{\alpha}^{\bullet}&=t_{\alpha}'+2=2(q_{\alpha}'+1) \\ t_{\alpha+1}^{\bullet}&=1=q_{\alpha+1} \\
			t_{\alpha+2}^{\bullet}&=2=2q_{\alpha+2} \\ t_{\alpha+3}^{\bullet}&=t_{\alpha+3}'+1=q_{\alpha+3}'+1.
			\end{align*}

			The number of incident quadrilaterals increases by $1$ only for the $4$ vertices $\alpha,\alpha+1,\alpha+2,\alpha+3$ in $D$ compared to $D'$, and hence the claimed equalities \eqref{eq:faces_vertices} hold for all vertices.
			\begin{figure}[htb]
				\begin{center}
					\includegraphics[scale=0.35, page=7]{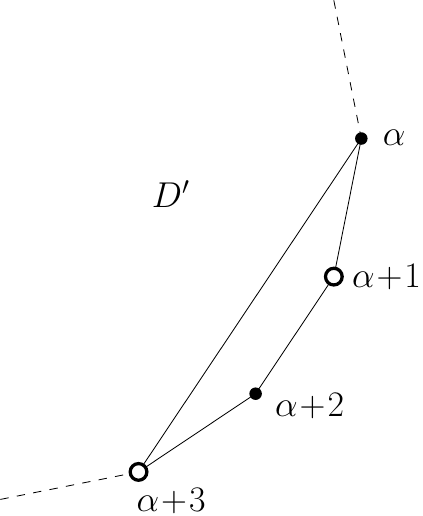}
					\caption{Cutting off an ear from $D$ gives a quadrangulation $D'$ on $2s$ vertices. Here $\alpha$ is odd.}\label{fig:cutting_quadrilateral} 
				\end{center}
			\end{figure}
			
			The argument for even $\alpha$ is similar.  
		\end{proof}
		\begin{rmk}
			Similarly, we obtain that for any quadrangulation $D$ and its associated triangulation $T_D^{\circ}$, the following equalities hold:
			\begin{align}
			t_{\alpha}^{\circ} = \left\{ \begin{matrix}
			q_{\alpha} & \text{ if } \alpha \text{ is odd} \\
			2q_{\alpha} & \text{ otherwise.}
			\end{matrix} \right.
			\end{align} 	 
		\end{rmk}
		\begin{prop}\label{prop:friezes}
			Let $D$ be a quadrangulation of $P$ and $F_{D}$ the associated frieze of type $\Lambda_4$. If $T_D^{\bullet}$ is the associated triangulation of $D$  and $F_{T_D^{\bullet}}$ the corresponding Conway--Coxeter frieze, then  $F_{D}$ and $F_{T_D^{\bullet}}$ coincide in every second row.
		\end{prop}
		\begin{proof}
			Set $\lambda:= \lambda_4$ and let the number of vertices of the polygon be $2s+2$.
			By Lemma~\ref{lem:number_triangles} the quiddity sequence of $T_D^{\bullet}$ for $D$ is
			\begin{align*}
			(t_0^{\bullet},t_1^{\bullet},t_2^{\bullet},t_3^{\bullet},\ldots,t_{2s}^{\bullet},t_{2s+1}^{\bullet})&=(q_0,2q_1,q_2,2q_3,\ldots,q_{2s},2q_{2s+1}).
			\end{align*}
			Let $(m_{i,j}^{\bullet})_{i,j\in \mathbbm{Z}}$ be the entries of the Conway--Coxeter frieze $F_{T_D^{\bullet}}$ defined by this quiddity sequence. Let $(m_{i,j})_{i,j\in \mathbbm{Z}}$ be the entries of the frieze $F_D$ of type $\Lambda_4$ with quiddity sequence $m_{\alpha,\alpha}=\sqrt{2}q_{\alpha}$ for $0\leq \alpha \leq 2s+1$. This yields
			\begin{align}\label{eq:quiddity_correspondence}
			m_{i,i}^{\bullet}=\left\{ \begin{matrix}
			\sqrt{2}m_{i,i} & \text{ if } \alpha \text{ is odd} \\
			\tfrac{1}{\sqrt{2}}m_{i,i} & \text{ otherwise.}
			\end{matrix} \right.
			\end{align}
			We have to show that for $i\in \mathbbm{Z}$ and $k\geq 0$,
			\begin{align}\label{eq:odd_rows}
			m_{i,i+2k+1}=m_{i,i+2k+1}^{\bullet}.
			\end{align}
			In order to prove this identity, we show that
			\begin{align}\label{eq:even_rows}
			m_{i,i+2k}^{\bullet}=\left\{ \begin{matrix}
			\sqrt{2} m_{i,i+2k} & \text{ if } i \text{ is odd} \\
			\tfrac{1}{\sqrt{2}}m_{i,i+2k} & \text{ otherwise.}
			\end{matrix} \right.
			\end{align}
			This holds for $k=0$ by equation~\eqref{eq:quiddity_correspondence} for all $i\in \mathbbm{Z}$.
			
			We prove equations~\eqref{eq:odd_rows} and \eqref{eq:even_rows}  by induction on $k$.
			Let $k=0$. By the unimodular rule we obtain that
			\begin{align*}
			m_{i,i+1}=2q_iq_{i+1}-1=m_{i,i}^{\bullet}m_{i+1,i+1}^{\bullet}-m_{i+2,i+1}^{\bullet}=m_{i,i+1}^{\bullet}.
			\end{align*}  
			Now assume that for $k$ equations~\eqref{eq:odd_rows} and \eqref{eq:even_rows} are fulfilled. The induction hypothesis and the unimodular rule yield
			\begin{align*}
			m_{i,i+2k+2}^{\bullet}=\frac{m_{i,i+2k}^{\bullet}m_{i+1,i+2k+1}^{\bullet}-1}{m_{i+1,i+2k}^{\bullet}}=\left\{ \begin{matrix}
			\sqrt{2}m_{i,i+2k+2} & \text{ if } i \text{ is odd} \\
			\tfrac{1}{\sqrt{2}}m_{i,i+2k+2} & \text{ otherwise.}
			\end{matrix} \right.
			\end{align*}
			Finally, by using the unimodular rule again, we find
			\begin{align*}
			m_{i,i+2k+3}^{\bullet}&=\frac{m_{i,i+2k+1}^{\bullet}m_{i+1,i+2k+2}^{\bullet}-1}{m_{i+1,i+2k+1}^{\bullet}}=\\
			&=\frac{\tfrac{1}{\sqrt{2}}\sqrt{2}m_{i,i+2k+1}m_{i+1,i+2k+2}-1}{m_{i+1,i+2k+1}}=m_{i,i+2k+3}
			\end{align*}
			for all $i\in \mathbbm{Z}$ by periodicity. As $(m_{i,j}^{\bullet})_{i,j\in \mathbbm{Z}}$ are entries of a Conway--Coxeter frieze, we conclude that the entries of $F_D$ are positive integers in every second row.	
		\end{proof}
		\begin{rmk}\label{rmk:triangulation_circ}
			The proof for $(m_{i,j}^{\circ})_{i,j\in \mathbbm{Z}}$ uses similar arguments. Using the same setting as in the proof of Proposition~\ref{prop:friezes}, the only difference concerns the entries of the rows $(m_{i,i+2k}^{\circ})_{i\in \mathbbm{Z},k\in \mathbbm{Z}_{\geq 0}}$, as
			\begin{align*}
			m_{i,i+2k}^{\circ}=\left\{ \begin{matrix}
			\tfrac{1}{\sqrt{2}} m_{i,i+2k} & \text{ if } i \text{ is odd} \\
			\sqrt{2} m_{i,i+2k} & \text{ otherwise.}
			\end{matrix} \right.
			\end{align*}
		\end{rmk}
		\begin{prop}\label{prop:uniqueness}
			Let $D$, $T_D^{\bullet}$ and $T_D^{\circ}$ be defined as above. Then $F_{T_{D}^{\bullet}}$ and $F_{T_{D}^{\circ}}$ are the only friezes, which coincide with $F_{D}$ in every second row. 
		\end{prop}
		\begin{proof}
			Assume that there exists a triangulation $T\notin \{T_D^{\bullet}, T_D^{\circ}\}$ of the $2s+2$-gon such that $F_T=:(m_{i,j}')_{i,j\in \mathbbm{Z}}$ and $F_D=(m_{i,j})_{i,j\in \mathbbm{Z}}$ coincide in every second row. Then in particular
			\begin{align}\label{eq:further_triangulation}
			m_{i,i+1}'=m_{i,i+1} \Leftrightarrow m_{i,i}'=\frac{2q_iq_{i+1}}{m_{i+1,i+1}'}
			\end{align}
			for all $i\in \mathbbm{Z}$. The entries $m_{i,i}'$ have to be positive integers. Furthermore, they have to be multiples (or divisors) of $q_{i,i}$ by equation~\eqref{eq:further_triangulation}.  As $T\neq T_D^{\bullet}$, there exists an index $0\leq j \leq s$ s.t. 
			\begin{align*} 
			m_{2j,2j}'\neq m_{2j,2j}^{\bullet}=q_{2j} \text{  or  } m_{2j+1,2j+1}'\neq m_{2j+1,2j+1}^{\bullet}=2_{q_{2j+1}}.
			\end{align*}
			In both cases, this entry of $F_T$ can either be greater or less than the entry of $F_{T_D^{\bullet}}$.
			As the entries of the Conway--Coxeter frieze are positive integers, we have the following cases for some integer $k>1$:
			\begin{itemize}
				\item[(a)] $m_{2j,2j}'=k\cdot q_{2j}$ (and hence $m_{2j+1,2j+1}'=\tfrac{2q_{2j+1}}{k}$)
				\item[(b)] $m_{2j,2j}'=\tfrac{q_{2j}}{k}$ (and hence $m_{2j+1,2j+1}'= 2k\cdot q_{2j+1}$).
			\end{itemize}
			In both cases, all entries $m_{2i,2i}'$ and  $m_{2i+1,2i+1}'$ ($i\in \mathbbm{Z}$) obey these relations simultaneously via equation~\eqref{eq:further_triangulation}. 
			
			Let $\{\alpha,\alpha+1,\alpha+2,\alpha+3\}$ be an ear of $D$. Then $q_{\alpha+1}=q_{\alpha+2}=1$ and both cases~(a) and (b) give a contradiction, as the quiddity sequence of $(m_{i,j}')_{i,j\in \mathbbm{Z}}$  would have a non-integer entry $m_{\alpha+2,\alpha+2}'=\tfrac{2}{k}$ (case~(a)) or $m_{\alpha+1,\alpha+1}'=\tfrac{1}{k}$ (case~(b)) for odd $\alpha$ and  $m_{\alpha+1,\alpha+1}'=\tfrac{2}{k}$ (case~(a)) for $m_{\alpha+2,\alpha+2}'=\tfrac{1}{k}$ (case~(b)) for even $\alpha$.
			It only remains to consider $k=2$ in the case~(a), as case~(b) gives a contradiction, as $\tfrac{1}{2}$ is not an integer. Case~(a) gives the entries of the Conway--Coxeter frieze of $T_D^{\circ}$, which is also a contradiction.
		\end{proof}
		Proposition~\ref{prop:friezes} and Proposition~\ref{prop:uniqueness} give Theorem~\ref{thm:main} for $p=4$.

		\section{Friezes of type $\Lambda_6$ and associated Conway--Coxeter friezes.}\label{sec:friezes_type_6}
		We use similar arguments as in the previous section, but we have to associate different triangulations to a $6$-angulation $D$ of a polygon $P$ with $4s+2$ vertices. 
		
		We colour even vertices white and odd vertices black again. Then every diagonal of $D$ connects a black and a white vertex. We triangulate every $6$-gon $P_j$ by inserting edges between all pairs of black vertices (or white vertices respectively). An example is shown in Figure~\ref{fig:hexagonal_dissection}, where the bold edges are diagonals of $D$.
		
		\begin{figure}[!htb]
			\begin{center}
				\includegraphics[scale=0.3, page=1]{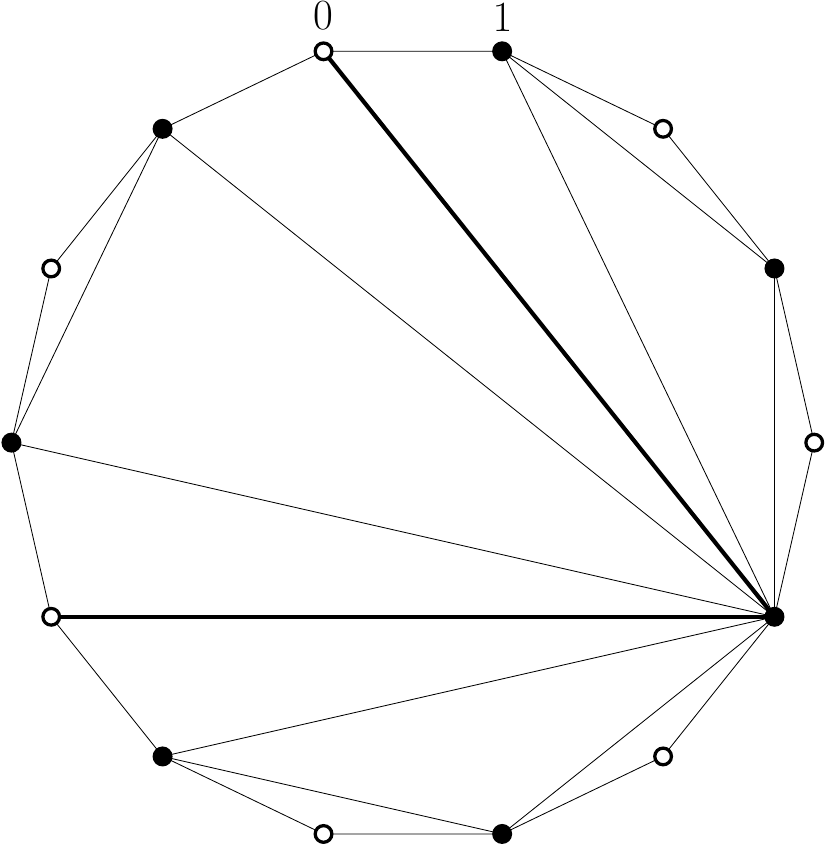}
				\hspace{5 ex}
				\includegraphics[scale=0.3, page=2]{hexagonal}
				\caption{The triangulations $T_{D}^{\bullet}$ (left) and $T_{D}^{\circ}$ (right) associated to a $6$-angulation $D$ (bold diagonals).}\label{fig:hexagonal_dissection} 
			\end{center}
		\end{figure}
		
		Let  $q_{\alpha}$ be the number of $6$-gons incident with a vertex $0\leq \alpha \leq 4s+1$ in $D$ and $t^{\bullet}_{\alpha}$ (or $t^{\circ}_{\alpha}$)  be the number of triangles incident with vertex $\alpha$ as in the previous section. The quiddity sequences of the frieze of type $\Lambda_6$ of the $6$-angulation $D$ of Figure \ref{fig:hexagonal_dissection} (bold diagonals) and of $T_D^{\bullet}$ and $T_{D}^{\circ}$  are
		\begin{align*}
		(2\sqrt{3}, \sqrt{3}, \sqrt{3}, \sqrt{3}, \sqrt{3},  3\sqrt{3}, \sqrt{3}, \sqrt{3}, \sqrt{3}, \sqrt{3}, 2\sqrt{3}, \sqrt{3}, \sqrt{3}, \sqrt{3})
		\end{align*}
		and
		\begin{align*}
		(2, 3, 1, 3, 1, 9, 1, 3, 1, 3, 2, 3, 1, 3), \ (6, 1, 3, 1, 3, 3, 3, 1, 3, 1, 6, 1, 3, 1)
		\end{align*}
		respectively.
		We obtain by the same arguments as in Lemma~\ref{lem:number_triangles}, that for $T_D^{\bullet}$ and $T_D^{\circ}$ and $0\leq \alpha \leq 4s+1$, 
		\begin{align*}
		t_{\alpha}^{\bullet} = \left\{ \begin{matrix}
		q_{\alpha} & \text{ if } \alpha \text{ is even} \\
		3q_{\alpha} & \text{ otherwise}
		\end{matrix} \right. , \ 
		t_{\alpha}^{\circ} = \left\{ \begin{matrix}
		q_{\alpha} & \text{ if } \alpha \text{ is odd} \\
		3q_{\alpha} & \text{ otherwise.}
		\end{matrix} \right.
		\end{align*} 
		Defining $\lambda := \lambda_6=2 \cos\left(\frac{\pi}{6} \right)= \sqrt{3}$, and proving that the entries $(m_{i,j})_{i,j\in \mathbbm{Z}}$ of the Conway--Coxeter frieze $F_{T_D^{\bullet}}$ fulfil the equations
		\begin{align*}
		m_{i,i+2k}^{\bullet}=\left\{ \begin{matrix}
		\sqrt{3} m_{i,i+2k} & \text{ if } i \text{ is odd} \\
		\tfrac{1}{\sqrt{3}}m_{i,i+2k} & \text{ otherwise}
		\end{matrix} \right. , \ 
		m_{i,i+2k}^{\circ}=\left\{ \begin{matrix}
		\sqrt{3} m_{i,i+2k} & \text{ if } i \text{ is even} \\
		\tfrac{1}{\sqrt{3}}m_{i,i+2k} & \text{ otherwise}
		\end{matrix} \right.
		\end{align*}
		in the same way as in Proposition~\ref{prop:friezes} and Remark~\ref{rmk:triangulation_circ}, we obtain that the friezes $F_D$ of type $\Lambda_6$ and the Conway--Coxeter frieze $F_{T_D^{\bullet}}$ (and $F_{T_D^{\circ}}$) of the assigned triangulation $T_D^{\bullet}$ (and $T_D^{\circ}$ respectively) coincide in every second row.
		
		In order to show that these are the only triangulations with that property, we have to consider a frieze $(m_{i,j}')_{i,j\in \mathbbm{Z}}$ arising from a triangulation $T\notin  \{T_D^{\bullet}, T_D^{\circ} \}$ which coincides with $F_D$ in every second row. By similar arguments as in Proposition~\ref{prop:uniqueness}, we have to show that the cases
		\begin{itemize}
			\item[(a)] $m_{2j,2j}'=k\cdot q_{2j}$ (and hence $m_{2j+1,2j+1}'=\tfrac{3q_{2j+1}}{k}$)
			\item[(b)] $m_{2j,2j}'=\tfrac{q_{2j}}{k}$ (and hence $m_{2j+1,2j+1}'= 3k\cdot q_{2+1}$)
		\end{itemize}
		where $j\in \mathbbm{Z}$ do not give the quiddity sequence of a Conway--Coxeter frieze beside $(m_{i,i}^{\bullet})$ and $(m_{i,i}^{\circ})$ (where $0\leq i \leq 4s+1$). If $k\notin\{1,3\}$, we get non-integer values for some of the entries $m_{\alpha_2,\alpha_2}',\ldots m_{\alpha_5,\alpha_5}'$ which belong to an ear $\{\alpha_1,\ldots,\alpha_6\}$ of $D$, which is a contradiction. 
		For $k=1$, both cases yield the quiddity row of $F_{T_D^{\bullet}}$ and for $k=3$, the first case gives the quiddity row of $F_{T_D^{\circ}}$ and the second case yields a contradiction, as  $\tfrac{1}{3}$ is not an integer. Hence Theorem~$\ref{thm:main}$ is fulfilled for $p=6$.
		
		We close this section by showing that Theorem~$\ref{thm:main}$ cannot hold for any other value of $p$.  We have to prove the following lemma. 
		\begin{lem}\label{lem:power_of_lambda_p}
			Let $p>3$ and $p\notin \{4,6\}$. Then for no integer $m\geq 1$, $\lambda_p^m$ is an integer.  
		\end{lem}
		\begin{proof}
			Assume to the contrary that $\left(2 \cos(\tfrac{\pi}{p})\right)=\sqrt[m]{l}$ for a positive integer~$l$. As the two numbers are equal, their field extensions $\mathbbm{Q}\left(2\cos\left(\tfrac{\pi}{p} \right)\right)$ and $\mathbbm{Q}\left(\sqrt[m]{l}\right)$ coincide. It is well known that the degree of the algebraic field extension $\left[\mathbbm{Q}\left(2\cos\left(\tfrac{\pi}{p} \right)\right):\mathbbm{Q}\right]$ equals $\tfrac{\varphi(2p)}{2}$, where $\varphi$ denotes Euler's $\varphi$ function. For any element $\sigma$ in the Galois group of the field extension, the equation  $\sigma(\sqrt[m]{l})=\varepsilon \sqrt[m]{l}$ holds, where $\varepsilon$ is an $m$-th root of unity. As  $\mathbbm{Q}\left(\sqrt[m]{l}\right)$ is a real field extension, $\varepsilon \in \{\pm 1\}$, and hence $\left[\mathbbm{Q}\left(\sqrt[m]{l}\right):\mathbbm{Q}\right]\leq2$. This yields $\varphi(2p)\leq 4$, and therefore $p\in \{2,3,4,5,6\}.$ In order to show that $p\neq 5$, we use the well known fact that $2\cos\left(\tfrac{\pi}{5}\right)=\Phi\notin \mathbbm{Z}$, where $\Phi$ denotes the golden ratio. Let $F_m$ be the $m$-th Fibonacci number, defined by initial values $F_0=0$ and $F_1=1$ and the recurrence relation $F_{m+1}=F_m+F_{m-1}$. Then the golden ratio obeys the relation  $\Phi^m=\Phi \cdot F_m+F_{m-1}$ for $m\geq 1$. As by definition every Fibonacci number is a positive integer, the equation for $\Phi^m$ implies that the term $\left(2\cos\left(\tfrac{\pi}{5}\right) \right)^m$ is not integral for any $m\geq 1$.  	
		\end{proof}
		In particular $\lambda_p^2$ is not integral for  $p\notin \{4,6\}$ by the lemma above. Therefore the entries of the second non-trivial row, which are of the form $\lambda_p^2k-1$ for some integer $k\geq 1$, are also not integral.
		
		\section{Farey graph $\mathscr{F}_p$ and friezes of type $\Lambda_p$}\label{sec:farey_graph}
		In this section we give an interpretation of the results of \cite{hj_2017} in the language of \cite{m-got_2015} in order to give a formula to determine the entries of the friezes of type $\Lambda_p$ from the values of the vertices of the corresponding  $p$-angulated path in the Farey graph $\mathscr{F}_p$ and vice versa.
		
		In order to do so, we recall necessary ingredients from  \cite[Section 5.2]{hj_2017}. For more details, see  \cite[Section 5.1]{hj_2017} and \cite[Section 2]{sw_2016}.
		
		We work in the (completed) hyperbolic plane $\overline{\mathbbm{H}}=\{z\in \mathbbm{C} \mid \Ima(z)\geq 0\}\cup\{\infty\}$. Let $L=\{iy \mid 0 \leq y < \infty \}\cup \{\infty \}$ be a line in $\overline{\mathbbm{H}}$ through the origin. An ideal point of $\overline{\mathbbm{H}}$ is a point on the horizontal axis. 
		
		Let $p\geq 3$ an integer and let $G_p$ be the discrete subgroup of the group of all M{\"o}bius transformations of $\overline{\mathbbm{H}}$ generated by $\sigma$ and $\tau_p$, for $\sigma(z)= -\tfrac{1}{z}$ and $\tau_p(z)=z+\lambda_p$. 
		
		Then the Farey graph $\mathscr{F}_p$ has vertices given by the $G_p$-orbit of $\infty$, and edges by the $G_p$-orbit of $L$. It is the skeleton of a tiling $\overline{\mathbbm{H}}$ by ideal $p$-gons, that is, $p$-gons whose vertices are ideal points. 
		
		Let $n \geq 0$ be an integer and $q_0$,$ \ldots $,$q_{n+2}$ positive integers. Define M{\"o}bius transformations by
		\begin{align*}
		\xi_{\alpha}(z)=q_{\alpha}\lambda_p -\frac{1}{z}.
		\end{align*} 
		Then $\xi_{\alpha}=\tau_p^{q_{\alpha}}\circ \sigma$ is in $G_p$, so 
		\begin{align}\label{eq:Farey_vertices}
		v_{\alpha}=\left\{ \begin{matrix}
		\infty & \alpha=0 \\
		\xi_0 \cdots \xi_{\alpha-1}(\infty) & 1\leq \alpha \leq n+2.
		\end{matrix} \right.
		\end{align}
		are vertices of $\mathscr{F}_p$ for $1\leq \alpha \leq n+2$.
		
		In \cite[Section 5]{hj_2017}, the authors prove the following: 
		\begin{lem}\label{lem:farey_general}
			Assume that $\xi_0 \cdots \xi_{n+2}(\infty)=\infty$.
			\begin{itemize}
				\item[(i)] The vertices $v_0,\ldots,v_{n+2}$ form a closed path in $\mathscr{F}_p$.
				\item[(ii)] Let $P'$ be the  full subgraph of $\mathscr{F}_p$ defined by $v_0,\ldots,v_{n+2}$. Then $P'$ is an $(n+3)$-gon with a $p$-angulation which divides $P'$ into $p$-gons $P_1',\ldots,P_s'$ such that
				\begin{align*}
				q_{\alpha}= | \{P_i'\mid P_i' \text{ is incident with } v_{\alpha} \}|
				\end{align*}  
				for $0\leq \alpha \leq n+2$.
			\end{itemize}
		\end{lem}
		Recall that the full subgraph on $v_0,\ldots, v_{n+2}$ is obtained by deleting all other vertices and the edges incident with them in $\mathscr{F}_p$.
		
		Figure~\ref{fig:farey_quadrangulation} shows (a subgraph of) the Farey graph $\mathscr{F}_4$ and the Farey path $P'$ of $D$ shown in Figure~\ref{fig:quadrangulation} except for the vertex $v_0$.
		The vertices of the closed path (bold path in figure) are $v_0=\infty$, $v_1=\tfrac{\sqrt{2}}{1}$, $v_{2}=\tfrac{3}{2\sqrt{2}}$, $v_3=\tfrac{2\sqrt{2}}{3}$, $v_{4}=\tfrac{1}{\sqrt{2}}$,
		$v_5=\tfrac{\sqrt{2}}{3}$,  $v_{6}=\tfrac{3}{5\sqrt{2}}$, $v_7=\tfrac{2\sqrt{2}}{7}$,
		$v_{8}=\tfrac{1}{2\sqrt{2}}$ and $v_9=\tfrac{0}{1}$. The number $q_{\alpha}$ arises as the path $P'$ takes the $q_{\alpha}$-th right turn at vertex $v_{\alpha}$ and this coincides with the number of $p$-gons incident with $v_{\alpha}$. This is due to \cite[Section 2]{sw_2016}.  For instance, $q_{1}=2$, as $P_1'$ and $P_2'$ are incident with $v_1$. 
		
		\begin{figure}[!htb]
			\begin{center}
				\includegraphics[scale=0.6, page=3]{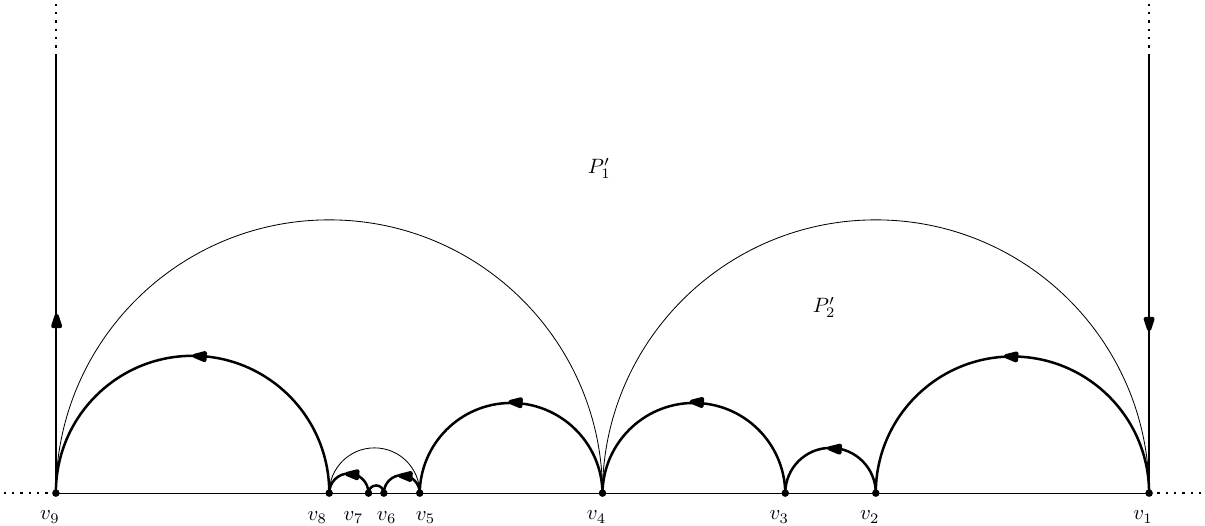}
				\caption{The Farey graph $\mathscr{F}_4$ is the skeleton of a tiling of
					$\overline{ \mathbbm{H} }$ by ideal $4$-gons like $P_1'$ and $P_2'$.}
				\label{fig:farey_quadrangulation}
			\end{center}
		\end{figure} 
		
		By identifying $v_{\alpha}$ with $\alpha$, the $p$-angulation of $P'$ becomes a $p$-angulation $D$ of $P$ which divides $P$ into $p$-gons $P_1$,$\ldots$,$P_s$.
		
		We need the following definition for stating the results:
		
		\begin{defn}
			The Farey distance of two vertices $v_{i}=\tfrac{a_i}{b_i}$ and $v_{j}=\tfrac{a_j}{b_j}$ of the Farey graph $\mathscr{F}_p$ is 
			\begin{align}
			d(v_i,v_j):= |a_i b_j - a_j b_i |.
			\end{align}	
		\end{defn}
		Note that the Farey distance does not satisfy the triangle inequality.
		
		For the rest of this section, let $(q_0\lambda_p,q_1\lambda_p,\ldots,q_{n+1}\lambda_p,q_{n+2}\lambda_p)$ be the quiddity sequence of the frieze of type $\Lambda_p$ corresponding to the $p$-angulation $D$ of the $n+3$-gon, $(m_{i,j})_{i,j\in \mathbb{Z}}$ denote the entries of this frieze and $P'$ with $v_1,\ldots, v_n$ the associated Farey path.
		
		In preparation for Proposition~\ref{prop:farey_frieze} we show the following result. 
		
		\begin{lem}\label{lem:farey}
			Let $0\leq \alpha \leq n+2$. The vertices $v_{\alpha}$ of $P'$ satisfy the equation
			\begin{align}
			v_{\alpha}=\frac{m_{0,\alpha-1}}{m_{1,\alpha-1}},
			\end{align}
		\end{lem}
		\begin{proof}
			Recall that
			\begin{align*}
			v_{\alpha}=\left\{ \begin{matrix}
			\infty & \alpha=0 \\
			\xi_0 \cdots \xi_{\alpha-1}(\infty) & 1\leq \alpha \leq n+2.
			\end{matrix} \right.,
			\end{align*} 
			where $\xi_{\alpha}(z)=q_{\alpha}\lambda_p - \frac{1}{z}$.\\
			We will show the statement by induction.
			
			Let $1 \leq \alpha \leq n+2$ be fixed. Note that 
			\begin{align*}
			\xi_{\alpha-1}(\infty)=m_{\alpha-1,\alpha-1}
			\end{align*} 
			and
			\begin{align*}
			\xi_{\alpha-2}\xi_{\alpha-1}(\infty)&=\xi_{\alpha-2}(m_{\alpha-1,\alpha-1})=m_{\alpha-2,\alpha-2}-\frac{1}{m_{\alpha-1,\alpha-1}}\\
			&=\frac{m_{\alpha-2,\alpha-2}m_{\alpha-1,\alpha-1}-1}{m_{\alpha-1,\alpha-1}}=\frac{m_{\alpha-2,\alpha-1}}{m_{\alpha-1,\alpha-1}}.
			\end{align*}
			For $1\leq k <\alpha-1$, assume that
			\begin{align*}
			\xi_{\alpha-1-k}\cdots \xi_{\alpha-1}(\infty)=\frac{m_{\alpha-1-k,\alpha-1}}{m_{\alpha-1-k+1,\alpha-1}}.
			\end{align*}
			As shown above, this holds for $k=1$.
			Then, as using the induction hypothesis for $k=\alpha-2$ gives the first equality, we further obtain that
			\begin{align*}
			\xi_{0}\xi_{1}\cdots\xi_{\alpha-1}(\infty)=\xi_{0}\left(\frac{m_{1,\alpha-1}}{m_{2,\alpha-1}}\right)=m_{0,0}-\frac{1}{\frac{m_{1,\alpha-1}}{m_{2,\alpha-1}}}=\frac{m_{0,0}m_{1,\alpha-1}-m_{2,\alpha-1}}{m_{1,\alpha-1}}.
			\end{align*}
			By equation~\eqref{eq:relation_1} for $i=0$ and $j=\alpha-1$, we get that
			\begin{align*}
			v_{\alpha}= \xi_{0}\cdots\xi_{\alpha-1}(\infty)=\frac{m_{0,\alpha-1}}{m_{1,\alpha-1}}
			\end{align*}
			as desired. Note that $v_{n+2}=\tfrac{m_{0,n+1}}{m_{1,n+1}}=\tfrac{0}{1}=0$.
			
			We can extend this formula to $\alpha=0$ as
			\begin{align*}
			\frac{m_{0,-1}}{m_{1,-1}}=\frac{1}{0} 
			\end{align*} 
			and hence the term coincides with $v_{0}=\infty$.
		\end{proof}
		The denominators $e_{1,\searrow}$ and numerators $e_{0,\searrow}$, which give the vertices $v_{\alpha}$ for $0\leq \alpha \leq n+2$ can hence be read "diagonally" by the corresponding frieze of type $\Lambda_p$ of $D$,
		\begin{align*}
		e_{0,\searrow}&=(1,m_{0,0},m_{0,1}\ldots,m_{0,n-1},1,0) \\
		e_{1,\searrow}&=(0,1,m_{1,1},m_{1,2},\ldots, m_{1,n},1).
		\end{align*}   
		For instance, the quadrangulation $D$ shown in Figure~\ref{fig:quadrangulation} gives the following sequence of numerators 
		\begin{align*}
		(1,\sqrt{2},3,2\sqrt{2},1,\sqrt{2},3,2\sqrt{2},1,0)
		\end{align*} 
		and denominators 
		\begin{align*}
		(0,1,2\sqrt{2},3,\sqrt{2},3,5\sqrt{2},7,2\sqrt{2},1).
		\end{align*} 
		We can now state a general formula for the entries of the friezes of type $\Lambda_p$ corresponding to $P'$. Indices are considered modulo $n+3$.
		\begin{prop}\label{prop:farey_frieze}
			The entries of the frieze of type $\Lambda_p$ and width $n$ corresponding to $D$ fulfil the equation
			\begin{align*}
			m_{i,j}=d(v_{i-1},v_{j+1}),
			\end{align*} 
			where $v_{0},\ldots, v_{n+2}$ are the vertices of $P'$.   
		\end{prop}
		\begin{proof}
			The value of $d(v_{i-1},v_{j+1})$ is by Lemma~\ref{lem:farey}
			\begin{align*}
			d(v_{i-1},v_{j+1}) = d\left(\frac{m_{0,i-2}}{m_{1,i-2}},\frac{m_{0,j}}{m_{1,j}}\right) = | m_{0,i-2}m_{1,j} - m_{0,j}m_{1,i-2}|=m_{i,j}
			\end{align*}
			where the last equality follows  via \eqref{eq:relation_2}. By symmetry properties of the frieze, this holds for either $i<j$ and $i\geq j$.
		\end{proof}
		
		\begin{cor}
			The Farey distance between $v_{\alpha-1}$ and $v_{\alpha+1}$ is $q_{\alpha}\lambda_p$. 
		\end{cor}
		\begin{proof}
			This follows immediately by Proposition~\ref{prop:farey_frieze} with $i=j=\alpha$.
		\end{proof}
		
		Finally, we compare the values of the vertices of the Farey paths corresponding to friezes of type $\Lambda_4$ and $\Lambda_6$ to the associated Farey paths corresponding to Conway--Coxeter friezes described in Section~\ref{sec:friezes_type_4} and \ref{sec:friezes_type_6} respectively.
		
		Let $p=4$ or $p=6$ and $D$ be a $p$-angulation of a polygon with $n+3$ vertices. Note that $n+3$ is even by Remark \ref{rmk:vertices_of_p-angulation}. Let $T_D^{\bullet}$ be an associated triangulation. Then  $(t_0^{\bullet}\lambda_p, \tfrac{t_1^{\bullet}}{\lambda_p}, \ldots, t_{n+1}^{\bullet}\lambda_p, \tfrac{t_{n+2}^{\bullet}}{\lambda_p})$ is the quiddity sequence of the frieze of type $\Lambda_p$ associated to $T_D$. Denote by  $P_T'$ the corresponding triangulated Farey path of $\mathscr{F}_3$. Let $m_{i,j}^{\bullet}$ be the entries of the Conway--Coxeter frieze of $T_D^{\bullet}$ and let $w_{0},\ldots,w_{n+2}$ be the vertices of $P_T'$.
		\begin{prop}
			The vertices of $P_T'$ satisfy the equation
			\begin{align*}
			w_{\alpha} = \frac{1}{\lambda_p}v_\alpha
			\end{align*}
			for $0\leq \alpha \leq n+2$ and $p=4$ or $p=6$.
		\end{prop}
		\begin{proof}
			Proposition~\ref{prop:friezes} and the corresponding arguments in Section~\ref{sec:friezes_type_6} yield
			\begin{align*}
			v_{2i+1}=\frac{m_{0,2i}}{m_{1,2i}}= \frac{\lambda_p m_{0,2i}^{\bullet}}{m_{1,2i}^{\bullet}}=\lambda_p w_{2i+1}
			\end{align*}
			for odd vertices and
			\begin{align*}
			v_{2i}=\frac{m_{0,2i-1}}{m_{1,2i-1}}= \frac{m_{0,2i-1}^{\bullet}}{\frac{1}{\lambda_p} m_{1,2i-1}^{\bullet}}=\lambda_p w_{2i}
			\end{align*} 
			for even vertices, with $0\leq i \leq \frac{n+3}{2}$.
		\end{proof}
		Similarly, we obtain this result for the associated triangulation $T_D^{\circ}$.
		
		Figure~\ref{fig:farey_triangulation} shows (a subgraph of) the Farey graph $\mathscr{F}_3$ and the Farey path $P_T'$ of $T_{D}^{\bullet}$, shown in Figure~\ref{fig:triangulation_quadrangulation}, except for the vertex $w_0$.
		The vertices of the closed path (bold path in figure) are $w_0=\infty$, $w_1=\tfrac{1}{1}$, $w_{2}=\tfrac{3}{4}$, $w_3=\tfrac{2}{3}$, $w_{4}=\tfrac{1}{2}$,
		$w_5=\tfrac{1}{3}$,  $w_{6}=\tfrac{3}{10}$, $w_7=\tfrac{2}{7}$,
		$w_{8}=\tfrac{1}{4}$ and $w_9=\tfrac{0}{1}$. The number $t_{\alpha}$ arises as the path $P_T'$ takes the $t_{\alpha}$-th right turn at vertex $w_{\alpha}$ and this coincides with the number of triangles incident with $v_{\alpha}$.  For instance, $q_{1}=4$, as $T_1'$, $T_2'$, $T_3'$ and $T_4'$ are incident with $w_1$. 
		
		\begin{figure}[!htb]
			\begin{center}
				\includegraphics[scale=0.6, page=2]{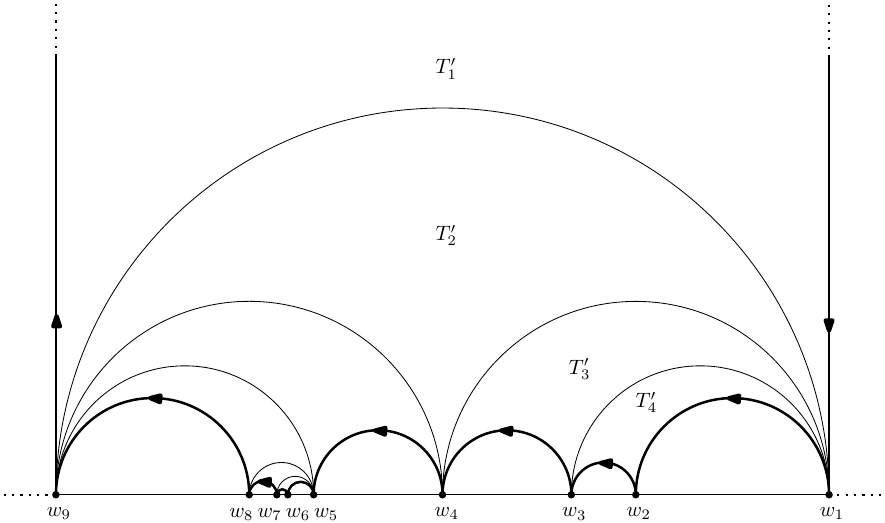}
				\caption{The Farey graph of the triangulation $T_{D}^{\bullet}$  shown in Figure~\ref{fig:triangulation_quadrangulation}.}
				\label{fig:farey_triangulation}
			\end{center}
		\end{figure}

		
		\paragraph{Acknowledgements}  The author is supported by the Austrian Science Fund (FWF): W1230, Doctoral Program ‘Discrete Mathematics’. 
		The author thanks his supervisor Karin Baur for her great advice and support and Peter J{\o}rgensen as well as anonymous referees for helpful comments on earlier versions. Furthermore, he thanks Florian Kainrath for providing the key ingredients for the proof of Lemma~\ref{lem:power_of_lambda_p}. The main work on this paper was done while the author was a guest at the School of Mathematics at the University of Leeds, England. He is deeply grateful to Robert~J.~Marsh for the invitation.

		\bibliographystyle{plain}
		\bibliography{andritsch}

	\end{document}